\documentclass[12pt]{amsart}
\usepackage{amssymb,amscd,amsthm,amsmath,color, xcolor}
\usepackage{fullpage}
\newcommand{\PP}{\mathbb{P}}
\newcommand{\OO}{\mathcal{O}}
\newcommand{\II}{\mathcal{I}}

\newcommand{\HH}{\mathcal{H}}

\newcommand{\CC}{\mathbb{C}}
\newcommand{\ZZ}{\mathbb{Z}}
\newcommand{\bF}{\mathbb{F}}
\newcommand{\eE}{\mathcal{E}}

\newcommand{\xX}{\mathcal{X}}
\newcommand{\yY}{\mathcal{Y}}
\newcommand{\kK}{\mathcal{K}}
\newcommand{\ver}{\operatorname{vert}}

\newcommand{\Pic}{\operatorname{Pic}}

\theoremstyle{plain}
\newtheorem{lemma}{Lemma}[section]
\newtheorem*{theorem*}{Theorem}
\newtheorem*{lemma*}{Lemma}
\newtheorem*{proposition*}{Proposition}
\newtheorem*{conjecture*}{Conjecture}
\newtheorem*{corollary*}{Corollary}
\newtheorem*{problem*}{Problem}
\newtheorem{theorem}[lemma]{Theorem}

\newtheorem{corollary}[lemma]{Corollary}
\newtheorem{proposition}[lemma]{Proposition}

\theoremstyle{definition}
\newtheorem{definition}[lemma]{Definition}
\newtheorem{example}[lemma]{Example}
\newtheorem{remark}[lemma]{Remark}

\begin{document}

\title{Algebraic hyperbolicity of very general surfaces}
\author[I. Coskun]{Izzet Coskun}
\address{Department of Mathematics, Statistics and CS \\University of Illinois at Chicago, Chicago, IL 60607}
\email{coskun@math.uic.edu}

\author[E. Riedl]{Eric Riedl}
\email{ebriedl@nd.edu}
\address{Department of Mathematics \\ University of Notre Dame, Notre Dame, IN 46556}

\thanks{During the preparation of this article the first author was partially supported by the  NSF grant DMS-1500031 and NSF FRG grant  DMS 1664296 and  the second author was partially supported  by the NSF RTG grant DMS-1246844.}

\begin{abstract}
Recently, Haase and Ilten initiated the study of classifying algebraically hyperbolic surfaces in toric threefolds. We complete this  classification for $\PP^1 \times \PP^1 \times \PP^1$, $\PP^2 \times \PP^1$, $\bF_e \times \PP^1$ and the blowup of $\PP^3$ at a point, augmenting our earlier work on $\PP^3$. In the process, we codify several different techniques for proving algebraic hyperbolicity, allowing us to prove similar results for hypersurface in any variety admitting a group action with dense orbit.
\end{abstract}

\maketitle

\section{Introduction}
 A complex projective variety $X$ is {\em algebraically hyperbolic} if there exists an ample divisor $H$ and a real number $\epsilon > 0$ such that the geometric genus $g(C)$ of any integral curve $C \subset X$ satisfies the inequality $$2g(C) -2 \geq  \epsilon \deg_H(C).$$ Recently, Haase and Ilten  \cite{HaaseIlten} initiated the study of classifying algebraically hyperbolic surfaces in toric threefolds. In this paper, we complete their classification for $\PP^1 \times \PP^1 \times \PP^1$, $\PP^2 \times \PP^1$, $\bF_e \times \PP^1$ and the blowup of $\PP^3$ at a point. Building on the results of Clemens \cite{Clemens86}, Ein \cite{Ein, Ein2}, Pacienza \cite{Pacienza}, Voisin \cite{Voisin, Voisincorrection}, Clemens-Ran \cite{ClemensRan}, Xu \cite{Xu,Xu2} and the authors \cite{CoskunRiedlhyperbolicity}, we develop a general technique for proving the algebraic hyperbolicity of the vanishing locus of a very general section of a sufficiently ample vector bundle. We apply our techniques to several other examples.  
 
 A complex manifold $X$ is {\em Kobayashi hyperbolic} if the Kobayashi pseudometric is nondegenerate and {\em Brody hyperbolic} if every entire map $f: \CC \rightarrow X$ is constant. By Brody's theorem \cite{Brody}, Kobayashi and Brody hyperbolicity agree for compact complex  manifolds. See Demailly's survey paper \cite{Demaillynew} for background on hyperbolicity, and \cite{RiedlYang, BercziKirwan} for recent advances on the Kobayashi Conjecture. Demailly introduced algebraic hyperbolicity as an algebraic analogue for Kobayashi hyperbolicity. He proved that for smooth projective varieties Kobayashi hyperbolicity implies algebraic hyperbolicity and conjectured the converse \cite{Demailly}. Deep conjectures of Green, Griffiths, Lang and Vojta predict that hyperbolicity has strong implications on the geometry and arithmetic of a variety \cite{CoskunHyperbolicity, Demailly}.  Hence, conjecturally, algebraic hyperbolicity is expected to control the geometry and arithmetic of $X$. Despite its importance, showing that a variety is algebraically hyperbolic can be very challenging. Nevertheless, we prove the following theorem answering many open cases of Question 6.6 from \cite{HaaseIlten}.

 \begin{theorem} Let $X$ be a very general surface in a threefold $A$ of class $D$. 
 \begin{enumerate}
\item  Let $A=\PP^1 \times \PP^1 \times \PP^1$ and $D=\sum_{i=1}^3 a_i H_i$, where $H_i$ is the pullback of the hyperplane class from the $i$th factor. Then $X$ is algebraically hyperbolic if and only if  up to permutations  $$a_1, a_2, a_3 \geq 3; \quad \mbox{or}  \quad a_1 = 2, a_2, a_3 \geq  4.$$ 
 \item Let $A= \PP^2 \times \PP^1$ and  $D= aH_1 + bH_2$, where $H_i$ is the pullback of the hyperplane class from the $i$th factor.  Then $X$ is algebraically hyperbolic if and only if  $$a \geq 4, b \geq 3 \quad \mbox{or} \quad b=2, a \geq 5.$$
 \item  Let $A$ be the blowup of $\PP^3$ at a point and $D=aH - bE$, where $H$ is the pullback of the hyperplane class  and $E$ is the class of the exceptional divisor. Then $X$ is algebraically hyperbolic if and only if $$a \geq b+2, b \geq 4 \quad \mbox{ or} \quad  b=0, a \geq 5.$$ 
 \item Let $A = \bF_e \times \PP^1$ and $D = a_1E + a_2 F + a_3 H$, where $\bF_e = \PP (\OO_{\PP^1} \oplus \OO_{\PP^1}(e))$, $E$ and $F$ are the pullbacks of the exceptional divisor and fiber class from $\bF_e$ and $H$ is the pullback of $\OO(1)$ from $\PP^1$. Let $\delta_{i,j}$ denote the Kronecker delta function. Then $X$ is algebraically hyperbolic if and only if 
 $$a_2-ea_1 \geq 2, a_1 \geq 3, a_3 \geq 3 \quad \mbox{ or} \quad a_3 = 2, a_1 \geq 4, a_2 - ea_1 \geq 3$$ 
 $$\mbox{or } a_1 = 2, a_3 \geq 4, a_2-ea_1 \geq 2 + \delta_{e,1}.$$
\end{enumerate}
 \end{theorem}
\noindent Our techniques apply to many other examples and yield almost complete classifications for many other varieties such as  $A= \PP(1, 1, 1, n)$. 
 
We improve and streamline a technique developed by  Ein \cite{Ein, Ein2}, Pacienza \cite{Pacienza}, Voisin \cite{Voisin, Voisincorrection} and  \cite{CoskunRiedlhyperbolicity}. Let $\eE$ be a globally generated vector bundle and let $X$ be the zero locus of a very general section of $\eE$. Assume that $X$ is smooth and irreducible. Let $h: C \rightarrow X$ be a birational map from a smooth curve $C$ into $X$. There is a close relation between the normal bundle $N_{h/X}$ to the map $h$ and the genus of $C$. The main idea is to quantify the positivity of $N_{h/X}$ using the tangent bundle $T_{\xX}$ to the universal zero locus  over the family of sections of $\eE$. We find maps from direct sums of appropriate Lazarsfeld-Mukai bundles $M_L$ onto $T_{\xX}$. When $M_L$ maps generically surjectively to $N_{h/X}$, we obtain a genus bound. We introduce the notion of section-dominating collection of line bundles to capture the positivity (see Definition \ref{def-sectionDominating}). Our results for threefolds can be summarized in the following theorem (see Corollary \ref{cor-main}).

\begin{theorem}\label{thm-intro3fold}
Let $A$ be a threefold with a Zariski open subset $A_0$ which is homogeneous under an algebraic group action $G$. Let $\eE$ be a very ample line bundle on $A$ invariant under $G$ and let $X$ be the zero locus of a very general section of $\eE$. Let $L_1, \dots, L_k$ be a collection of section-dominating line bundles for $\eE$ on $A$. Then for any genus $g$ curve $f:C \to X$ intersecting $A_0$, $$2g-2-K_X \cdot C = \deg N_{f/X} \geq - \max_i \{ \deg L_i|_C \}.$$ In particular, if there exists $\epsilon >0$ such that $K_X \cdot C \geq (1+\epsilon) \deg  L_i|_C$ for all curves $C$ and indices $i$ and $(A\setminus A_0) \cap X$ does not contain any curves of  genus $0$ or $1$, then $X$ is  algebraically hyperbolic.
\end{theorem}

Theorem \ref{thm-intro3fold} is essentially the same as Theorem 3.6 from \cite{HaaseIlten}, and the proofs are fairly similar. Haase and Ilten focus on the special case of toric varieties and use the language of connected sections and focal loci, while our results apply to arbitrary threefolds with dense orbit of a group action use the concept of section-dominating line bundles. However, the spirit of both techniques is the same, and is a continuation of work of Ein, Voisin, and others.

The straightforward application of this technique does not always produce sharp bounds, and indeed, Haase and Ilten include a list of open cases at the end of their paper where their Theorem 3.6 and our Theorem \ref{thm-intro3fold} do not prove algebraic hyperbolicity (see Question 6.6 of \cite{HaaseIlten}). In this paper, we resolve almost all of their open questions by using a careful analysis of where various line bundles fail to be section-dominating. For instance, we study the geometric consequences of the condition that the map from $M_L$ to $N_{h/X}$ is generically zero (see Lemma \ref{lem-unionOfFibers}). We also study possible quotients of Lazarsfeld-Muaki bundles and produce scrolls that contain the curve (see Lemma \ref{lem-scrolls}) to allow us to obtain the sharp result in other cases.

\subsection*{The organization of the paper} In \S \ref{sec-preliminaries}, we introduce our setup, prove general positivity criteria for $N_{h/X}$ and obtain genus bounds in terms of the positivity. In \S \ref{sec-examples}, we apply our techniques to explicit examples and classify algebraic hyperbolicity of the very general surface in $\PP^1 \times \PP^1 \times \PP^1$, $\PP^2 \times \PP^1$, $\bF_e \times \PP^1$ and the blowup of $\PP^3$ at a point. We also discuss several other examples.

\subsection*{Acknowledgements} We would like to thank Lawrence Ein, Christian Haase, Nathan Ilten and Mihai P\u{a}un for valuable conversations. 

\section{Preliminaries}\label{sec-preliminaries}
In this section, we improve and streamline certain positivity techniques developed by Ein  \cite{Ein, Ein2}, Pacienza \cite{Pacienza},   Voisin \cite{Voisin, Voisincorrection} and \cite{CoskunRiedlhyperbolicity}. We then use these techniques to obtain a criterion for showing the algebraic hyperbolicity of the zero locus of a very general section of a vector bundle on a variety $A$ admitting a group action with dense orbit.

\subsection{Setup} Let $A$ be a smooth, complex projective variety of dimension $n$ and assume that $A$ contains a Zariski-open set $A_0$ admitting a transitive group action by an algebraic group $G$.  For example, $A$ could be a homogeneous variety or a smooth, projective toric variety. Let $\eE$ be a globally generated  vector bundle invariant under $G$ on $A$ of rank $r < n-1$. Let $V = H^0(A, \eE)$. Assume that the zero locus $X$ of a very general section of $\eE$ is a smooth, irreducible variety. We wish to understand when $X$ is an algebraically hyperbolic variety. 

Suppose that $X$ contains a curve of degree $e$ and genus $g$ that meets $A_0$. Then if $\xX_1 \to V$ is the universal hypersurface over $V$, we have the relative Hilbert scheme $\HH \to V$ with universal curve $\yY_1 \to \xX_1$, where the general fiber of $\yY_1 \to \HH$ is a geometric genus $g$ curve of degree $e$. There is a natural $G$-action on $\HH$. By a standard argument, we can find a $G$-invariant subvariety $U \subset \HH$ such that the map $U \to V$ is \'{e}tale. Restricting $\yY_1$ to $U$, we get a family $\yY_2 \to U$. By taking a resolution of the general fiber (and possibly further restricting $U$), we get a smooth family $\yY \to U$ whose fibers are smooth curves of genus $g$. We can pull back $\xX_1$ to a family $\xX$ over $U$, with projection maps $\pi_1: \xX \to U$ and $\pi_2: \xX \to A$. We have a natural generically injective map $\yY \to \xX$ which we denote by $h$. 

Because $\yY$ was constructed to be stable under the $G$-action, we have that $\pi_2 \circ h$ dominates $A_0$. Let the {\em vertical tangent sheaf} $T_{\xX/A}$ be defined by the natural sequence
$$0 \to T_{\xX/A} \to T_{\xX} \to \pi_2^* T_A \to 0.$$ Let the {\em vertical tangent sheaf} $T_{\yY/A}$ be the kernel of the natural map $T_{\yY} \rightarrow h^* \pi_2^* T_A$, which is surjective if $A=A_0$, but may fail to be surjective otherwise.
Let $M = M_{\eE}$ denote the Lazarsfeld-Mukai bundle defined by the short exact sequence
\[ 0 \to M \to V \otimes \OO_A \to \eE \to 0 . \]
 Let $t$ be a general element of $U$. Let $Y_t$ be the fiber of $\yY$ over $t$ and $X_t$ be the fiber of $\xX$ over $t$. Let $h_t : Y_t \to X_t$ be the restriction of $h$ to $Y_t$. The following result generalizes \cite[Lemmas 2.2, 2.4 and 2.5]{CoskunRiedlhyperbolicity}.

\begin{proposition}\label{prop-summary}

\begin{enumerate} 
\item $N_{h_t/X_t} \cong N_{h/\xX}|_{Y_t}$.
\item $T_{\xX/A} \cong \pi_2^* M$.
\item If $A_0 = A$, then $N_{h/\xX}$ is the cokernel of the map of vertical tangent spaces $T_{\yY/A} \to T_{\xX/A}$. If $A_0 \neq A$, then the cokernel of the map $T_{\yY/A} \to T_{\xX/A}$ is a sheaf $\kK$ that injects into $N_{h/\xX}$ with torsion cokernel.
\end{enumerate}
\end{proposition}
\begin{proof}
Part (1) follows from the proof of \cite[Lemma 2.2]{CoskunRiedlhyperbolicity}. Part (2) follows from the proof of \cite[Lemma 2.4]{CoskunRiedlhyperbolicity} 

For Part (3), if $A$ is homogeneous under the action of $G$, then the map $T_{\yY} \to h^* \pi_2^* T_A$ is surjective. If $A \neq A_0$, then the map $T_{\yY} \to h^* \pi_2^*  T_A$ is not necessarily surjective---it is only generically surjective. Let $S$ be the image of $T_{\yY}$ in $h^* \pi_2^* T_A$. Then we have the following diagram:
\catcode`\@=11
\newdimen\cdsep
\cdsep=3em

\def\cdstrut{\vrule height .25\cdsep width 0pt depth .12\cdsep}
\def\@cdstrut{{\advance\cdsep by 2em\cdstrut}}

\def\arrow#1#2{
  \ifx d#1
    \llap{$\scriptstyle#2$}\left\downarrow\cdstrut\right.\@cdstrut\fi
  \ifx u#1
    \llap{$\scriptstyle#2$}\left\uparrow\cdstrut\right.\@cdstrut\fi
  \ifx r#1
    \mathop{\hbox to \cdsep{\rightarrowfill}}\limits^{#2}\fi
  \ifx l#1
    \mathop{\hbox to \cdsep{\leftarrowfill}}\limits^{#2}\fi
}
\catcode`\@=12

\cdsep=3em
$$
\begin{matrix}
& & 0 & & 0 & & 0 \cr
& & \arrow{u}{} & & \arrow{u}{} & & \arrow{u}{} \cr
0 & \arrow{r}{} &  S & \arrow{r}{} & h^* \pi_2^*T_{A} & \arrow{r}{} & T & \arrow{r}{} & 0 \cr
& & \arrow{u}{} & & \arrow{u}{} & & \arrow{u}{} \cr
0 & \arrow{r}{} &  T_{\yY}  & \arrow{r}{} & h^* T_{\xX} & \arrow{r}{} & N_{h/\xX} & \arrow{r}{} & 0          \cr
& &  \arrow{u}{} & & \arrow{u}{} & & \arrow{u}{} \cr
0 & \arrow{r}{} & T_{\yY/A} & \arrow{r}{} & h^*T_{\xX/A} & \arrow{r}{} & \kK & \arrow{r}{} & 0 \cr
& & \arrow{u}{} & & \arrow{u}{} & & \arrow{u}{} \cr
& & 0  & & 0 & & 0 \cr
\end{matrix}
$$
Here, $T$ is a torsion sheaf. Now restrict the above to a general curve parameterized by $\yY$. The second column remains exact since it is a sequence of vector bundles. The first column remains exact because $T_{\yY/A}$ is torsion free and the map $T_{\yY/A} \to T_{\yY}$ is generically injective when restricted to $C$. The rows remain exact by Lemma 2.1 of \cite{CoskunRiedlhyperbolicity}. Thus, by the nine lemma, the last column remains exact as well. Since $T$ is torsion, its restriction to a general curve  remains torsion.
\end{proof}

\subsection{Bounding the genus of curves} The following well-known lemma explains our interest in the positivity of $N_{h_t/X_t}$.

\begin{lemma}\label{lem-degreeGenus}
The following hold:
$$\deg(N_{h_t/X_t})= 2 g(Y_t) -2 - K_{X_t} \cdot h_t(Y_t),$$
$$\deg(N_{h_t/X_t}) \geq \deg(\kK|_{Y_t}).$$
\end{lemma}

\begin{proof}
The first equality follows by taking degrees of the terms in the exact sequence defining $N_{h_t/X_t}$
$$0 \rightarrow T_{Y_t} \rightarrow h_t^* T_{X_t} \rightarrow N_{h_t/X_t} \rightarrow 0.$$
Proposition \ref{prop-summary} provides a surjection from $h^* \pi_2^* M$ onto $\kK$, which in turn injects into $N_{h/\xX}$. Restricting to $Y_t$ we obtain a surjection onto a free subsheaf of $N_{h_t/X_t }$ of the same rank. It follows that the degree of $N_{h_t/X_t}$ is at least the degree of $\kK|_{Y_t}$.
\end{proof}

 It remains to bound the degree of $\kK|_{Y_t}$. We do this by working with bundles simpler than $M$ that still surject onto $\kK$. 

\begin{definition}\label{def-sectionDominating}
Let $\eE$ be a vector bundle on $A$. We say a non-trivial globally generated line bundle $L$ is a \emph{section-dominating} line bundle for $\eE$ if $\eE \otimes L^{\vee}$ is globally generated and the map 
\[ H^0(L \otimes \II_p) \otimes H^0(\eE \otimes L^{\vee}) \to H^0(\eE \otimes \II_p) \]
is surjective for every point $p \in A$. More generally, a collection of non-trivial globally generated line bundles  $L_1, \dots, L_u$ is a \emph{section-dominating  collection} of line bundles for $\eE$ if $\eE \otimes L_i^{\vee}$ is globally generated for every $1 \leq i \leq u$ and the map 
\[ \bigoplus_{i=1}^u (H^0(L_i \otimes \II_p) \otimes H^0(\eE \otimes L_i^{\vee})) \to H^0(\eE \otimes \II_p) \] is surjective for every point $p\in A$.
\end{definition}

\begin{example} Suppose $A = \PP^2 \times \PP^1$ with $H_i$ the pullback of the hyperplane class from the respective factors. Let $\eE = aH_1 + bH_2$ with $a,b > 0$. Then $H_1$ and $H_2$ are a section-dominating collection for $\eE$. Choose coordinates so that $p$ is $([0,0,1], [0,1])$. Then $H^0((aH_1+bH_2) \otimes \II_p)$ is the set of polynomials of bidegree $(a,b)$ in variables $x,y,z$ and $s,t$ where each monomial is divisible by either $x,y$ or $s$. This is precisely the image of the natural map $$H^0(H_1 \otimes \II_p) \otimes H^0((a-1)H_1+bH_2) \oplus H^0(H_2 \otimes \II_p) \otimes 
H^0(aH_1 + (b-1)H_2) \to H^0((aH_1+bH_2) \otimes \II_p).$$ 
\end{example}

\begin{example}
In a similar way, on $\PP^1 \times \PP^1 \times \PP^1$, $H_1$, $H_2$, and $H_3$ is a section-dominating collecting for any line bundle $aH_1+bH_2+cH_3$ with $a,b,c > 0$.
\end{example}

Given a section $\phi \in H^0(\eE \otimes L^{\vee})$, the natural multiplication map defines a map $L_i \to \eE$ and $H^0(L_i) \to H^0(\eE)$, hence induces a map from $M_{L_i} \to M_{\eE}$. 

\begin{proposition}
\label{prop-smallBundle}
Let $\eE$ be a globally generated vector bundle and $M_{\eE}$ the Lazarsfeld-Mukai bundle associated to $\eE$. Let $L_1, \dots, L_u$ be a section-dominating collection of line bundles for $\eE$. Then there is a surjection $\bigoplus_{i=1}^u M_{L_i}^{\oplus s} \to M_{\eE}$ for some integer $s$.
\end{proposition}
\begin{proof}
Every section of $\eE \otimes L^{\vee}$ induces a map $M_{L_i} \to M_{\eE}$. Let $s_i = h^0(\eE \otimes L_i^{\vee})$. If we choose a basis for $H^0(\eE \otimes L_i^{\vee})$, we obtain a map $M_{L_i}^{\oplus s_i} \to M_{\eE}$. Set
$s= \max_{1 \leq i \leq u} s_i$. We obtain a map $\bigoplus_{i=1}^u M_{L_i}^{\oplus s} \to M_{\eE}$. The fiber of $M_{\eE}$ at a point $p$ is the vector space of sections of $\eE$ vanishing at $p$. Similarly, the fiber of $M_{L_i}$ at a point $p$ is the vector space of sections of $L_i$ vanishing at $p$. Since by assumption,  \[ \bigoplus_{i=1}^u (H^0(L_i \otimes \II_p) \otimes H^0(\eE \otimes L_i^{\vee})) \to H^0(\eE \otimes \II_p) \] is surjective for every point $p\in A$, we conclude that this map is surjective. 
\end{proof}

\subsection{The case of threefolds} Throughout this subsection we assume that  $A$ is a threefold with a Zariski open subset $A_0$ which is homogeneous under an algebraic group action $G$. Let $\eE$ be a very ample line bundle on $A$ invariant under $G$ and let $X$ be the zero locus of a very general section of $\eE$. Let $\xX$ and $\yY$ be the corresponding universal families and let $\kK$ be the quotient of $h^* T_{\xX/A}$ by $T_{\yY/A}$  constructed in the setup above. Since $A$ is a threefold, $\kK|_C$ must be rank 1, which simplifies our proofs and allows us to apply Proposition \ref{prop-smallBundle} to get a genus bound for curves on a very general section of $\eE$. We start with the following observation.

\begin{proposition}
\label{cor-degImBundle}
Given a surjection from $M_L$ to a line bundle $N$ on a curve $C$, we have that $\deg N \geq - \deg L|_C$.
\end{proposition}
\begin{proof}
We have the short exact sequence
\[ 0 \to M_L \to \OO \otimes H^0(L) \to L \to 0 .  \]
Taking the second wedge power of the sequence, we get
\[ 0 \to \wedge^2 M_L \to \wedge^2 (\OO \otimes H^0(L)) \to M_L(L) \to 0 .  \]
Thus, $M_L(L)$ is globally generated, and hence, so is $N(L)$. Since the degree of $N(L)$ must be non-negative, we have $\deg N \geq - \deg L|_C$.
\end{proof}

Since $A$ is a threefold, $X$ will be a surface, so the normal sheaf of a curve in $X$ will have rank 1. We can use the above discussion to obtain the following result. 

\begin{corollary}\label{cor-main}
Let $L_1, \dots, L_k$ be a collection of section-dominating line bundles for $\eE$ on $A$. Then for any genus $g$ curve $f:C \to X$ intersecting $A_0$, $$2g-2-K_X \cdot C = \deg N_{f/X} \geq - \max_i \{ \deg L_i|_C \}.$$ In particular, if there exists $\epsilon >0$ such that $K_X \cdot C \geq (1+\epsilon) \deg L_i|_C$ for all curves $C$ and indices $i$ and $(A\setminus A_0) \cap X$ does not contain any  curves of genus $0$ or $1$, then $X$ is  algebraically hyperbolic.
\end{corollary}

\begin{proof}
To prove the statement, it will suffice to prove that $\deg N_{f/X} \geq - \max_i \{ \deg L_i|_C \}.$ By Lemma \ref{lem-degreeGenus}, it is enough to bound $\deg \kK|_{C}$, where we write $C$ instead of the more cumbersome $Y_t$. 

By Proposition \ref{prop-summary} part (2), $\kK|_{C}$ admits a surjection from $M_{\eE}|_{C}$. By Proposition \ref{prop-smallBundle}, we get a surjection from a direct sum of copies of $M_{L_i}$ onto $M_{\eE}$, and hence a surjection of a direction sum of copies of $M_{L_i}|_{C}$ onto $\kK|_{C}$. Since $\kK|_{C}$ has rank $1$, we can find a single $i$ such that $\alpha: M_{L_i}|_{C} \to \kK|_{C}$ is generically surjective. By Proposition \ref{cor-degImBundle}, $\deg \alpha(M_{L_i}|_{C}) \geq - \deg L_i|_{C}$. Putting it all together, we get $$\deg N_{f/X} \geq \alpha(M_{L_i}|_{C}) \geq - \deg L_i|_{C}.$$
\end{proof}

\begin{remark}
As in Haase and Ilten's Lemma 3.7 \cite{HaaseIlten}, Corollary \ref{cor-main} can be extended to Gorenstein threefolds that admit a crepant resolution.
\end{remark}

In the examples we consider, we will run across situations where we do not have an appropriate collection of section-dominating line bundles. There are other useful techniques for these cases.

\begin{lemma}\label{lem-torsion}
Let $L$  be a line bundle on $A$ such that $H^0( L^{\vee} \otimes \eE) \not= 0$. Assume that  the natural map $\bigoplus_i M_{L} \to M_{\eE}$ is not surjective and that the induced map $\bigoplus_i M_{L} \to \kK$ has torsion image. Then at a general point $(p,t)$ of $\yY$, $(T_{\yY/A})_{(p,t)}$ contains the image of the map $$H^0(L \otimes \II_p) \otimes H^0( L^{\vee}\otimes \eE) \to H^0(\eE \otimes \II_p).$$
\end{lemma}
\begin{proof}
By Proposition \ref{prop-summary} (2),  $T_{\xX/A} \cong \pi_2^* M_{\eE}$. If at a general point $p$ of the curve $Y_t$, $(M_L)_p \to (M_{\eE})_p$ maps to $0$ in $\kK_p$, then  the image of $(\pi_2^* M_L)_p$ must lie in $(T_{\yY/A})_p$. The result follows.
\end{proof}

\begin{lemma}
\label{lem-unionOfFibers}
Let $p \in  A_0$ and let $Z = \pi_2^{-1}(p)$, where  $\pi_2: \xX \to A$. Let $T$ be a subvariety of $A$ containing $p$.  Assume that $H^0(\II_T \otimes \eE)$ is contained in $(T_{\yY/A})_{(p,f)}$ for general $(p,f) \in h^{-1}(Z)$. Then $W = h(\yY) \cap Z$ is a union of fibers of the map $\beta: Z \to H^0(\eE|_T)$ given by sending $(p,f)$ to $f|_{T}$.
\end{lemma}
\begin{proof}
There is a natural map $\beta : Z \rightarrow H^0(\eE|_{T})$ defined by sending a section $f$ of $\eE$ to the restriction of $f$ to $T$. The tangent space to a fiber of this map (at any point) is  $H^0(\II_{T} \otimes \eE)$. By generic smoothness, $Z$ is reduced because $\xX$ is. Since the map $\yY \to \xX$ is birational onto its image,  $h^{-1}(Z) \to Z$ is also  birational onto its image. Since $W = h(\yY) \cap Z$ is reduced and the fibers of $\beta$ sweep out $Z$, it follows from generic smoothness that either $W$ is a union of fibers of $\beta$ or that $W$ intersects a general fiber of $\beta$ transversely. However, $(T_{\yY/A})_{(p,f)}$ by assumption contains the tangent space to the fiber $H^0(\II_{T} \otimes \eE)$. Hence, $W$ does not intersect the general fiber of $\beta$ transversely, so the result follows.
\end{proof}

Finally, considering scrolls that contain $Y_t$ we can improve the genus bounds.  Suppose we have a nontrivial map from $M_L \to N_{h_t/X}$. Then the degree of $N_{h_t/X}$ is at least the degree of the image of this map, so need to  bound the possible degrees of $1$-dimensional quotients of $M_L$. This is closely related to the scrolls that $h_t(Y_t)$ lies on. 

\begin{lemma}
\label{lem-scrolls}
Let $L$ be a base-point-free line bundle on the threefold $A$ and let $\phi:A \rightarrow \PP^n$ be the morphism induced by the complete linear system $|L|$. 
 Then a rank one quotient $Q$ of $M_L|_{Y_t}$ induces a surface scroll $\PP(Q')$ over $Y_t$ and a map  $\PP(Q') \to \PP^n$ whose $\phi^*(\OO_{\PP^n}(1))$-degree is  equal to $\deg Q+\deg L|_{Y_t}$.
\end{lemma}
\begin{proof}
Consider the Euler sequence for $\PP^n$
\[ 0 \to \Omega_{\PP^n}(1) \to \OO_{\PP^n}^{n+1} \to \OO_{\PP^n}(1) \to 0 . \]
Pulling the sequence back to $Y_t$ and observing that $\Omega_{\PP^n}(1)$ is  $M_L|_{Y_t}$, we obtain
\[ 0 \to M_L|_{Y_t} \to \OO_{Y_t}^{n+1} \to L|_{Y_t} \to 0 . \]
The quotient $Q$ of $M_L|_{Y_t}$ gives rise to a subsheaf $S$ of $M_L|_{Y_t} \subset \OO_{Y_t}^{n+1}$. Taking the quotient of $\OO_{Y_t}^{n+1}$ by $S$ gives us a sheaf $Q'$ of degree equal to $\deg Q+\deg L|_{Y_t}$. The result follows from the universal property of projective space.
\end{proof}

\section{Examples}\label{sec-examples}
In this section, we apply the technique described in \S \ref{sec-preliminaries} to specific examples.

\subsection{A very general surface in $\PP^1 \times \PP^1 \times \PP^1$}
The variety $\PP^1  \times \PP^1 \times \PP^1$ admits three natural projections  $\pi_i$, $1 \leq i \leq 3$, to $\PP^1$. Let $H_i= \pi_i^* \OO_{\PP^1}(1)$.  Then $$\Pic(\PP^1  \times \PP^1 \times \PP^1) = \ZZ H_1 \oplus \ZZ H_2 \oplus \ZZ H_3 \quad \mbox{and} \quad K_{\PP^1  \times \PP^1 \times \PP^1}= -2H_1 - 2H_2 - 2H_3.$$ Let $a_1, a_2, a_3$ be nonnegative integers and let $X$ be a very general surface of class $\sum_{i=1}^3 a_i H_i$. By the generalized Noether-Lefschetz Theorem \cite{ravindraSrinivas}, if $a_i \geq 2$ for all $i$, then the natural restriction map $\Pic(\PP^1 \times \PP^1 \times \PP^1) \to \Pic(X)$ is an isomorphism. Hence, the class of any curve on $X$ can be written as $\OO_X(c_1, c_2, c_3)$. The degree of a curve $C$ on the class $\OO_X(c_1, c_2, c_3)$ with respect to the ample class $H= H_1 + H_2 + H_3$ is 
$$a_1 c_2 + a_1 c_3 + a_2 c_1 + a_2 c_3 + a_3 c_1 + a_3 c_2.$$

\begin{lemma}
Let $a_1 \leq a_2 \leq a_3$ be positive integers such that either  $a_1 =1$ or $a_1=2$ and $a_2 \leq 3$. Then a surface $X$ with class $\sum_{i=1}^3 a_i H_i$ is not algebraically hyperbolic.
\end{lemma}
\begin{proof}
If $a_1=1$, then the projection $\pi_{2,3}$ of $X$ onto the second and third factors gives a birational map between $X$ and $\PP^1 \times \PP^1$, hence $X$ is rational and not algebraically hyperbolic. If $a_1=a_2=2$, then the fibers of the projection $\pi_3$ exhibit a one-parameter family of elliptic curves on $X$. Finally, if $a_1=2$ and $a_2=3$, then the fibers of the projection $\pi_3$ exhibit a one-parameter family of $\OO_{\PP^1 \times \PP^1}(2,3)$ curves on $X$. Such a family necessarily has singular members. While the general member of the family has genus $2$,  the singular members have geometric genus  one or less. Hence, $X$ cannot be algebraically hyperbolic.  
\end{proof}

\begin{theorem}\label{thm-p1p1p1}
Let $X$ be a very general surface in $\PP^1 \times \PP^1 \times \PP^1$ of class $\sum_{i=1}^3 a_i H_i$. Assume that either $a_i \geq 3$ for all $i$ or that one of the $a_i$ is $2$ and the other two are at least $4$. Then $X$ is algebraically hyperbolic.
\end{theorem}
\begin{proof}
Let $M_{H_i}$ denote the Lazarsfeld-Mukai bundle defined by 
$$0 \to M_{H_i} \to \OO \otimes H^0(\PP^1 \times \PP^1 \times \PP^1, H_i) \to H_i \to 0.$$  Any section in $H^0(\PP^1 \times \PP^1 \times \PP^1, \OO(a_1, a_2, a_3) \otimes H_i^{\vee})$, defines a natural map $$M_{H_i} \to M_{\OO(a_1,a_2,a_3)} = T_{\xX}^{\ver},$$ in the notation of \S \ref{sec-preliminaries}. Since every tri-graded polynomial vanishing at a point can be written as a polynomial combination of polynomials of degree 1 in each set of variables vanishing at the point, we see that the map  $\bigoplus_s (M_{H_1} \oplus M_{H_2} \oplus M_{H_3}) \to T_{\xX}^{\ver}$ is surjective for a sufficiently large $s$. Hence, the normal bundle $N_{\yY/\xX}$ admits a generically surjective map from $M_{H_i}$ for some $i$. 

By permuting the indices if necessary, assume that $i=1$. Let $f: C \rightarrow X$ be the normalization of a curve on $X$ with class $\OO_X(c_1, c_2, c_3)$. Then, by Proposition \ref{cor-degImBundle},  the degree of $N_{f/X}$ is at least $-H_1 \cdot C$. Thus, we have
\[ 2g-2  -K_X \cdot C \geq \deg N_{f/X} \geq - H_1 \cdot C .\]
Hence, 
\[ 2g-2 \geq (a_1-3)(a_2c_3+a_3c_2) + (a_2-2)(a_1c_3+a_3c_1) + (a_3-2)(a_1c_2+a_2c_1) .\]

If  $a_i \geq 3$ for $1 \leq i \leq 3$, let $a_{\max} = \max\{a_i\}$ and $a_{\min} = \min\{a_i\}$. We have that $$H \cdot C \leq 2 a_{\max} \sum_i c_i.$$ Hence,
\[ 2g-2 \geq a_1 c_3 + a_3c_1 + a_2c_3+a_3c_2 \geq a_{\min} \sum_i c_i \geq \frac{a_{\min}}{2a_{\max}} C \cdot H . \]
Thus, $X$ is algebraically hyperbolic.

Now suppose $a_2 = 2$ and $a_1, a_3 \geq 4$. Then we have
\[ 2g-2 \geq a_2c_3+a_3c_2 + a_1c_2+a_2c_1 \geq a_{\min} \sum_i c_i \geq \frac{a_{\min}}{2a_{\max}} C \cdot H  . \] 
A similar argument applies if $a_3 =2$ and $a_1, a_2 \geq 4$.

The last remaining case is when $a_1 = 2$ and $a_2, a_3 \geq 4$. If the map from $M_{H_2}$ or $M_{H_3}$ to $N_{f/X}$ has non-torsion image, then we can apply the arguments in the previous two cases.  Therefore, we may assume that both $M_{H_2}$ and $M_{H_3}$ have torsion image in $N_{f/X}$. This imposes strong restrictions on the geometry of $C$.  In particular, we now show that $C$ then has to be the ramification locus of the double cover $\pi_{2,3}: X \rightarrow \PP^1 \times \PP^1$. 

By Lemma \ref{lem-torsion},  $(T_{\yY/A})_p$ contains all possible images of $(M_{H_2})_p$ and $(M_{H_3})_p$ under an arbitrary multiplication map.  Hence, $(T_{\yY/A})_p$ contains the ideal sheaf of the fiber of the projection $\pi_{2,3}$ passing through $p$. By Lemma \ref{lem-unionOfFibers},  $\yY$ is a union of fibers of the projection $\pi_{2,3}$.  If $Y_t$ intersects a fiber of $\pi_{2,3}$ in two distinct points, then $\yY$ would be dense in $\xX$ because  $\yY$ is invariant under the group action $\PP GL_2 \times \PP GL_2 \times\PP GL_2 $ and the orbit of two distinct points under $\PP GL_2$ is dense.  Hence, $C=Y_t$ must consist of the ramification locus of the double cover $\pi_{2,3} : X_t \to \PP^1 \times \PP^1$. The ramification locus of $\pi_{2,3}$  is isomorphic to a curve of class $(2a_2, 2a_3)$ in $\PP^1 \times \PP^1$ and has genus $(2a_2-1)(2a_3-1) > 1$. Hence, $X$ is algebraically hyperbolic in this case as well.
\end{proof}

We remark that this completely characterizes algebraic hyperbolicity of such surfaces.

\subsection{A very general surface  in $\PP^2 \times \PP^1$}
The variety $\PP^2 \times \PP^1$ admits two natural projections $\pi_1$ and $\pi_2$ to $\PP^2$ and $\PP^1$, respectively. Let $H_1 = \pi_1^* \OO_{\PP^2}(1)$ and let $H_2 =  \pi_2^* \OO_{\PP^1}(1)$. Then
$$\Pic(\PP^2 \times \PP^1) = \ZZ H_1 \oplus \ZZ H_2 \quad \mbox{and} \quad K_{\PP^2 \times \PP^1} = -3H_1 - 2H_2.$$  We let $H=H_1 + H_2$ be the very ample class defining  the Segre embedding. Let $X$ be a very general surface with class $aH_1 + b H_2$, where $a, b$ are nonnegative integers. Then $K_X = \OO_X(a-3, b-2)$.  By the generalized Noether-Lefschetz Theorem \cite{ravindraSrinivas}, if $a \geq 3$ and $b\geq 2$, then every curve $C$ on $X$ has class $\OO_X(c, d)$ for some $c, d \geq 0$. The $H$-degree of $C$ is given by $$\deg_H(C)= ac + ad + bc.$$

\begin{lemma}
Let $X$ be a very general surface in $\PP^2 \times \PP^1$ of class $aH_1 + bH_2$. If $b=1$ or $a \leq 3$ or $(a,b) = (4,2)$, then $X$ is not algebraically hyperbolic.
\end{lemma}
\begin{proof}
If $b=1$, then $X$ is birational to $\PP^2$, hence rational. If $a \leq 3$, then the projection onto $\PP^1$ gives a covering of $X$ by plane curves of degree $a \leq 3$. Since these curves are either rational (if $a \leq 2$) or elliptic curves (if $a=3)$, $X$ cannot be algebraically hyperbolic.

Now suppose that $a=4, b=2$. Then we can view $X$ as the vanishing of $gs^2+hst+kt^2$ where $g,h$ and $k$ are degree $4$ polynomials in the variables $x,y,z$ on $\PP^2$. Such $X$ can be viewed as a double cover of $\PP^2$ branched over the octic curve $B$ defined by $h^2-4gk=0$, which is smooth for a general choice of $g,h$ and $k$. The preimage of a bitangent line to $B$ is a geometric genus $1$ curve in $X$ with two nodes. It follows that such surfaces are not algebraically hyperbolic.
\end{proof}

\begin{theorem} \label{thm-p2p1}
Let $X$ be a very general surface in $\PP^2 \times \PP^1$ of class $aH_1 + b H_2$. Then $X$ is algebraically hyperbolic if and only if $$b\geq 3, a \geq 4 \quad \mbox{ or} \quad  b=2, a \geq 5.$$
\end{theorem}
\begin{proof}
We use the ample class $H= H_1 + H_2$. Let $f: C \to X$ be a curve on $X$ with $f(C)$ in the class $\OO_X(c,d)$. Then $\deg_H(C) = (a+b)c+a d.$  Since $a(c+d) \leq (a+b)c + ad \leq (a+b)(c+d),$ we have that  $$c+d \geq \frac{1}{a+b} \deg_H(C).$$ Hence, to prove algebraic hyperbolicity, it suffices to bound $2g(C) -2$ below by a multiple of $c+d$. As in the proof of Theorem \ref{thm-p1p1p1}, $H_1$ and $H_2$ form a section dominating collection of line bundles for $aH_1 + bH_2$. Hence, by Proposition \ref{prop-smallBundle}, $M_{H_1}$ or $M_{H_2}$ maps generically surjectively to $N_{f/X}$.  We discuss the possibilities separately. 

\textbf{Case 1:} $M_{H_1}$ maps generically surjectively to $N_{f/X}$ and $a \geq 5$, $b \geq 2$. Then we have
\[ 2g-2 -K_X \cdot C \geq -H_1 \cdot C ,\]
so that
\[ 2g-2 \geq ((a-4)H_1+(b-2)H_2) \cdot C = (a-4)(bc+ad)+(b-2)ac . \]
Since $a \geq 5, b \geq 2$, we are done.

\textbf{Case 2:} $M_{H_1}$ maps generically surjectively to $N_{f/X}$, $a=4$ and $b \geq 3$. Observe that in Case 2, $f^*M_{H_1}$ has a rank $1$ quotient bundle $Q$ of degree $q$, the degree of the torsion-free part of $N_{f/X}$. By Lemma \ref{lem-scrolls}, we obtain a scroll $\PP(Q') \to C$ of $H_1$ degree $q+H_1 \cdot C$. Let $S$ be the image of the scroll in $\PP^2 \times \PP^1$, and let $S = \alpha H_1 + \beta H_2$. From the description of the scroll $\PP(Q')$ in the proof of Lemma \ref{lem-scrolls}, we see that the fibers of $\PP(Q') \to C$ over a point $p$ map to lines in $\PP^2 \times \{ \pi_2(f(p)) \}$. Thus, $S \cdot H_2$ consists of one line for every point of $C \cdot H_2$, so $\alpha = H_2 \cdot C = ac$. Since $H_1^2 \cdot S = q+H_1 \cdot C$, we know that $\beta = q + H_1 \cdot C$. Because $C$ lies in $S$, it follows that $\beta \geq d$. Thus,
\[ \deg N_{f/X} \geq q = \beta - C \cdot H_1 \geq d - C \cdot H_1 . \]
It follows that
\[ 2g-2 - K_X \cdot C \geq d - C \cdot H_1 \]
or
\[ 2g-2 \geq d + (a-4)H_1 \cdot C + (b-2)H_2 \cdot C = d + (a-4)(bc+ad)+(b-2)ac. \]
If $b \geq 3$, then we see that this last is at least $c+d$, and we are done.

\textbf{Case 3:} $M_{H_2}$ maps generically surjectively to $N_{f/X}$ and $b \geq 3$, $a \geq 4$. Then we have
\[ 2g-2 -K_X \cdot C \geq - H_2 \cdot C,  \]
so that
\[ 2g-2 \geq ((a-3)H_1 + (b-3)H_2) \cdot C = (a-3)(bc+ad) + (b-3)(ac). \]
If $b \geq 3$, $a \geq 4$, then we  see that $$2g-2 \geq bc+ad \geq c+d.$$ 

\textbf{Case 4:} $M_{H_1}$ has torsion image in $N_{f/X}$ and $b=2$, $a \geq 5$. By Lemma \ref{lem-torsion}, we see that for a general $p \in \yY$, $(T_{\yY})_p$ contains the ideal sheaf of the fiber of the projection $\pi_1$ of $X$ onto the $\PP^2$ factor. By Lemma \ref{lem-unionOfFibers}, $\yY$ is a union of fibers of the projection map $\pi_1$. If a generic fiber of $\pi_1$ on $\yY$ were two distinct points, then $\yY$ would be dense in $\xX$ because $\yY$ is invariant under the $\PP GL_3 \times \PP GL_2$ action and the orbit of two distinct points under $\PP GL_2$ is dense. Hence, $f(C)$ is contained in the ramification locus of $\pi_1$. The ramification locus is isomorphic to a curve in $\PP^2$ of degree $2a$, which has genus larger than $1$ for the range of degrees we consider.

Thus, $X$ is always algebraically hyperbolic.

\end{proof}

This completely determines when $X$ is algebraically hyperbolic.

\subsection{A very general surface in $\bF_e \times \PP^1$} This example generalizes the previous two examples. Let $e \geq 0$ be a nonnegative integer. Let $\bF_e$ denote the Hirzebruch surface $\PP(\OO_{\PP^1} \oplus \OO_{\PP^1}(e))$. Since we have already discussed $\PP^1 \times \PP^1 \times \PP^1$, we assume $e \geq 1$. The complement of $E \times \PP^1 \cong \PP^1 \times \PP^1$ in $\bF_e \times \PP^1$ is homogeneous. The surface $\bF_e$ admits a natural projection to $\PP^1$. Let $F$ denote the class of a fiber and let $E$ denote the class of a section with self-intersection $-e$. Then $\Pic(\bF_e) = \ZZ E \oplus \ZZ F$. Let $\pi_2$ be the natural projection of $\bF_e \times \PP^1$ to $\PP^1$ and let $H = \pi_2^* \OO_{\PP^1}(1)$. We then have 
$$\Pic (\bF_e \times \PP^1) = \ZZ E \oplus \ZZ F \oplus \ZZ H \quad \mbox{and} \quad K_{\bF_e \times \PP^1} = -2E - (e+2) F - 2H.$$
Let $X$ be a very general surface of class $a_1E + a_2F + a_3H$. If $a_1, a_3 \geq 2, a_2 \geq e+2$, then by the generalized Noether-Lefschetz Theorem, $\Pic(X) \cong \Pic(\bF_e \times \PP^1)$ and every curve on $X$ has the class $c_1 E + c_2 F + c_3 H$. We set $L$ to be the very ample class $E + (e+1) F + H$. We then have 
$$\deg_L(C)= a_1c_2 + a_2 c_1 + a_2 c_3 + a_3 c_2 + a_1 c_3 + a_3 c_1 - e a_1 c_1.$$ As usual, let $a_{\min}$ and $a_{\max}$ be the minimum and the maximum of the $a_i$. 
Observe that $$\deg_L(C) \leq 2 a_{\max} \sum_{i=1}^3 c_i.$$ By adjunction,
$$K_X = (a_1-2)E + (a_2 - e-2)F + (a_3 -2)H.$$ 

\begin{lemma}
Suppose $X$ is as above. Consider the three numbers $a_1, a_3$ and $a_2' = a_2 - ea_1$. Then $X$ is not algebraically hyperbolic in any of the following situations 
\begin{itemize}
\item if any of $a_1, a_2'$ or $a_3$ is less than $2$.
\item if $a_3 = 2$ and one of $a_1$ and $a_2'$ is equal to $2$.
\item if $(a_1, a_3)$ is $(2,3)$ or $(3,2)$.
\item if $e=1$, $a_1 = 2$, $a_2' = 2$ and $a_3 \geq 1$.
\end{itemize}
Outside of these situations, the pullback of $E$ from $\bF_e$ does not contain any rational or genus 1 curve.
\end{lemma}
\begin{proof}
Consider the following two surfaces:
\begin{itemize}
\item The pullback of $E$ from $\bF_e$.
\item The pullback of $F$ from $\bF_e$.
\end{itemize}
Each of these surfaces is isomorphic to a $\PP^1 \times \PP^1$, and intersecting each surface with $X$, we obtain curves of classes $(a_2',a_3)$ and $(a_1,a_3)$ on the $\PP^1 \times \PP^1$. Using the fact that the genus of a curve on $\PP^1 \times \PP^1$ of class $(m,n)$ is $(m-1)(n-1)$, we obtain the result in the first two cases.

Now consider the third case. Consider the one-parameter family of the pullbacks of $F$ from $\bF_e$. This gives a 1-parameter family of genus 2 curves in a $\PP^1 \times \PP^1$. Since the discriminant locus in $\PP^1 \times \PP^1$ is ample, the family must contain singular elements having geometric genus at most 1. Thus, the surface is not algebraically hyperbolic in this case.

Finally, consider the last case. Intersecting $X$ with the $1$-parameter family of $\bF_1$'s corresponding to the fibers of the projection onto the $\PP^1$ factor, we get a $1$-parameter family of curves of class $2E+4F$ on $\bF_1$. Viewing $\bF_1$ as the blowup of $\PP^2$ at a point, we obtain a $1$-parameter family of quartic plane curves singular at the point of the blowup. Such curves generically have genus 2. This family of curves must meet the discriminant locus, so it must parameterize curves with extra singularities, which have genus at most 1. The result follows.
\end{proof}

\begin{theorem}
Let $e >0$ and let $X$ be a very general surface of class $a_1 E + a_2 F + a_3 H$ in $\bF_e \times \PP^1$. Let $a_2' = a_2 - ea_1$. Then $X$ is algebraically hyperbolic in the following cases:
\begin{itemize}
\item if $a_2' \geq 2$ and $a_1, a_3 \geq 3$
\item if $a_3 = 2, a_1 \geq 4,$ and $a_2' \geq 3$
\item if $a_1 = 2, a_3 \geq 4$ and $a_2' \geq 2 + \delta_{e,1}$ where $\delta_{e,1}$ is $0$ if $e \neq 1$ and $1$ if $e = 1$
\end{itemize}
\end{theorem}

\begin{proof}
We first consider the case when $a_2-ea_1 \geq 2$, $a_1, a_3 \geq 3$. The collection of line bundles $E+eF, F$ and $H$ are a section dominating collection of line bundles for $a_1 E + a_2 F + a_3 H$. Hence, by Proposition \ref{prop-smallBundle}, $M_{E+eF}$, $M_F$ or $M_{H}$ maps generically surjectively to $N_{f/X}$. We discuss the three possibilities separately. 

\textbf{Case 1:} $M_{E+eF}$ maps generically surjectively to $N_{f/X}$. Then 
$$2g-2 - K_X \cdot C \geq - (E+eF) \cdot C,$$ so that 
$$2g-2 \geq (a_2 - ea_1 +e-2)(a_1c_3+a_3 c_1) + (a_1-3)(a_2c_3 + a_3c_2)+ (a_3 -2)(a_1c_2 + a_2 c_1 -ea_1c_1)$$ 
$$\geq (2+3e)c_1 +3c_2 +  3e c_3  \geq \frac{3}{2a_{\max}} \deg_L(C).$$

\textbf{Case 2:} $M_{F}$ maps generically surjectively to $N_{f/X}$. Then 
$$2g-2 - K_X \cdot C \geq - F \cdot C,$$ so that 
$$2g-2 \geq (a_2 - ea_1 +e-3)(a_1c_3+a_3 c_1) + (a_1-2)(a_2c_3 + a_3c_2)+ (a_3 -2)(a_1c_2 + a_2 c_1-ea_1c_1)$$
$$\geq 2 c_1 + (a_1 + a_3) c_2 + a_2 c_3 \geq \frac{1}{a_{\max}} \deg_L(C).$$

\textbf{Case 3:} $M_{H}$ maps generically surjectively to $N_{f/X}$. Then 
$$2g-2 - K_X \cdot C \geq - H \cdot C,$$ so that 
$$2g-2 \geq (a_2 - ea_1 +e-2)(a_1c_3+a_3 c_1) + (a_1-2)(a_2c_3 + a_3c_2)+ (a_3 -3)(a_1c_2 + a_2 c_1-ea_1c_1)$$
$$ \geq 3e c_1 +3c_2 +  (3e +a_2)c_3 \geq \frac{3}{2a_{\max}}\deg_L(C).$$

We now consider the case $a_3 = 2$, $a_2-ea_1 \geq 3$, $a_1 \geq 4$. If either $M_F$ or $M_{E+eF}$ map to $N_{f/X}$ with non-torsion image, then we see by the above discussion that $\deg N_{f/X}$ is at least $c_1+c_2+c_3$, so it remains to consider the case where both bundles have torsion image. By Lemma \ref{lem-torsion} this occurs when $(T_{\yY/A})_p$ contains all possible images of $(M_F)_p$ and $(M_{E+eF})_p$ under a multiplication map. This implies that $(T_{\yY/A})_p$ contains the ideal sheaf of the fiber of the projection to $\bF_e$. By Lemma \ref{lem-unionOfFibers}, it follows that $\yY$ consists of a union of fibers of this projection map. If the curve is contained in $E$ (i.e., the exceptional curve of $\bF_e$  times $\PP^1$), then we already know that the curve has higher genus. We may assume that the curve is not contained in $E $. Since $\PP GL_2$ acts transitively on the set of pairs of distinct points in $\PP^1$, it follows that $f(C)$ must be contained in the ramification locus of the double cover. This means that $C$ is isomorphic to a curve of class $2a_1E + 2a_2F$ in $\bF_e$. Furthermore, for the general $X$ this curve is smooth. Thus, the genus of $C$ is at least
\[ (2a_1E+2a_2F) \cdot ((2a_1-2)E+(2a_2-2-e)F) = -2a_1(2a_1-2)e+2a_1(2a_2-2-e)+2a_2(2a_1-2) \]
\[ = 2a_1(2a_2-2a_1-e)+2a_2(2a_1-2) . \]
Since $a_2 \geq a_1e+3$, we see that the first term is positive, and hence the entire sum is greater than $1$. We conclude that $X$ is algebraically hyperbolic.

We now consider the case $a_1 = 2$, $a_3 \geq 4$, and $a_2 - ea_1 \geq 2+\delta_{e,1}$ where $\delta_{e,1} = 0$ if $e \neq 1$ and $1$ if $e = 1$. If either $M_F$ or $M_H$ maps to $N_{f/X}$ with non-torsion image, then we see that $\deg N_{f/X}$ is at least $c_1 + c_2 + c_3$ (in the case $e=1$, we get slightly worse bounds on $a_2$ coming from Case 2). It remains to consider the case where both bundles have torsion image. It follows from Lemma \ref{lem-torsion} that $(T_{\yY/A})_p$ contains the image of $(M_H)_p$ and $(M_F)_p$ in $(T_{\xX / A})_p$ under all possible multiplication maps. Let $\ell$ be a fiber of the projection of $\bF_e \times \PP^1 \to \bF_e$. Then the ideal sheaf of $\ell$ is contained in $(T_{\yY/A})_p$. By Lemma \ref{lem-unionOfFibers}, it follows that $\yY$ is a union of fibers of the map from $X$ to the exceptional divisor $E$. We may assume that the curve is contained in $E$.  Then, by the transitive group action, it follows that $\yY$ must be contained in the ramification locus of the double cover.  Thus, $C$ is isomorphic to a curve of class $(2(a_2-e), 2a_3)$ in $E \cong \PP^1 \times \PP^1$. Thus, the genus is always greater than 1 in this case.
\end{proof}

Note that this completely characterizes when such $X$ are algebraically hyperbolic.


\subsection{$\PP^3$ blown up at a single point}
Let $A$ be the blowup of $\PP^3$ at a point $q$. Let $H$ denote the pullback of the hyperplane class in $\PP^3$ and let $E$ be exceptional divisor over $q$. Then $$\Pic(A) = \ZZ H \oplus \ZZ E \quad \mbox{and} \quad K_A = -4H + 2E.$$ The effective cone of $A$ is spanned by $E$ and $H-E$ and the nef cone of $A$ is spanned by $H$ and $H-E$. In the chamber bounded by $E$ and $H$, the divisor $E$ is in the stable base locus. Consequently, any irreducible effective  divisor with class $aH - b E$ must either have $a=0, b=-1$ or $a \geq b \geq 0$. The exceptional divisor $E$ is isomporphic to $\PP^2$ and not algebraically hyperbolic, so let us restrict to the case $a \geq b \geq 0$. Let $X$ be a very general surface in the class $aH - bE$. If $a \geq 4$ and $b \geq 2$, then by the generalized Noether-Lefschetz Theorem \cite{ravindraSrinivas} any curve on $X$ has class $cH-dE$ restricted to $X$.

\begin{lemma}
Let $X$ be a surface in the blowup of $\PP^3$ at a point $q$ with class $aH - bE$. If $1 \leq b \leq 3$ or $a \leq b+1$, then $X$ is not algebraically hyperbolic.
\end{lemma}
\begin{proof}
The intersection of $X$ with the exceptional divisor $E$ is a plane curve of degree $b$. If $b \leq 3$, this curve is either rational or elliptic. Hence, $X$ cannot be algebraically hyperbolic. Observe that $a$ must be at least $b$ in order for $\OO(aH-bE)$ to have sections. If $a = b+1$, then the projection of $X$ from $p$ gives a birational map from $X$ to $\PP^2$, hence $X$ is rational. Finally, if $a=b$, then the image of $X$ in $\PP^3$ is a cone, and is hence covered by rational curves.
\end{proof}

We further remark that by \cite{CoskunRiedlhyperbolicity} if $b=0$ and $a \geq 5$, then the very general $X$ is algebraically hyperbolic. If $b=0$ and $a \leq 4$ it is well-known that $X$ contains rational curves and is not algebraically hyperbolic. The following theorem covers the remaining cases. 

\begin{theorem}
Let $X$ be a very general surface in the blowup of $\PP^3$ at a point $p$ with class $aH - bE$. If $a \geq b+2 \geq 6$, then $X$ is algebraically hyperbolic.
\end{theorem}

\begin{proof}
The complement of $E$ in $A$ is homogeneous. Since we have $b \geq 4$, the intersection of $X$ with the exceptional divisor $E$ is a curve of high genus. All other curves on $X$ intersect the complement of $E$, hence we can use our technique. We take $2H -E$ as our very ample class. The divisor classes $H$ and $H-E$ form a section dominating collection of line bundles for $aH - bE$. Suppose there is a curve $C$ with class $\OO_X(cH -dE)$ for $c \geq d \geq 0$ on $X$. We have $$\deg_{2H-E}(C) = 2ac -bd.$$ Since $H$ and $H-E$ are a section dominating collection, either $M_H$ or $M_{H-E}$ maps generically surjectively to $N_{f/X}$. 

\textbf{Case 1:} $M_H$ maps generically surjectively to $N_{f/X}$ and $a \geq b+3 \geq 7$. We have
$$2g -2 - K_X \cdot C \geq - H \cdot C,$$
or equivalently 
$$ 2g-2 \geq(a-5) ac - (b-2)bd.$$
If $a \geq b+3$, then we see that $$2g-2 \geq (b-2)((b+3)c - bd)  \geq (b-2) 3c \geq \frac{3(b-2)}{2a} \deg_{2H-E}(C).$$

\textbf{Case 2:} $M_{H-E}$ maps generically surjectively to $N_{f/X}$ and $a \geq b+2 \geq 6$. we have
$$2g -2 - K_X \cdot C \geq - (H-E) \cdot C,$$
or equivalently 
$$ 2g-2 \geq(a-5) ac - (b-3)bd \geq (b-3) ((b+2)c -bd) \geq (b-3) 2c \geq \frac{b-3}{a}\deg_{2H-E}(C).$$
It follows that $X$ is algebraically hyperbolic.

\textbf{Case 3:} $M_{H-E} \to N_{f/X}$ has torsion image and $a = b+2 \geq 6$. By Lemma \ref{lem-torsion}, this means that for a general point $(p,X) \in \yY$ that the map $(\oplus_s M_{H-E})_p \to T_{\xX / A}$ factors through $T_{\yY / A} \to T_{\xX / A}$. It follows that $(T_{\yY / A})_p$ contains the degree $b+2$ hypersurfaces in $\PP^n$ with a $b$-fold point at $q$ and containing the line $\ell$ from $p$ to $q$. Applying Lemma \ref{lem-unionOfFibers}, we see that since the family $\yY$ was constructed to be invariant under the automorphisms of $\PP^n$ fixing $p$ and $q$, $\ell$ must intersect $X$ set theoretically in just two points: $p$ and $q$, since otherwise, the automorphisms of $A$ would show that $\yY$ was dense in $\xX$.

Hence, these give us two curves on each $X$. If we can show that these curves do not have genus $0$ or $1$, we conclude that $X$ is algebraically hyperbolic.
\begin{enumerate}
\item The curve of points residual to a line meeting $X$ to order $b+1$ at $q$. This locus is a curve isomorphic to the curve $E \cap X$. For a general $X$,  this is a smooth plane curve of degree at least  4, hence has genus at least 3. 
\item The curve $D$ of points $x \in X$ such that the line joining $x$ and $q$ is tangent to $X$ at $x$. The curve $D$ is isomorphic to the branch locus of the projection $\pi$ of $X$ from $q$. The projection $\pi$ realizes $X$ as a double cover of $\PP^2$. Comparing canonical bundles, we compute the class of $X$
$$2 K_X =  \pi^*(2 K_{\PP^2} + [D]).$$  Hence, $D$ is a plane curve of degree $2b+2$. Since $b \geq 4$, this curve is a curve of degree at least $10$ in $\PP^2$.  Furthermore for a general $X$, $D$ is smooth. To see this, choose coordinates such that $q= [0:0:0:1]$. Assume that $X$ has equation $f_{b+2}(x,y,z) + f_{b+1}(x,y,z)w+ f_b(x,y,z) w^2$, where $f_i$ is a homogeneous polynomial in $x,y,z$ of degree $i$. Then the branch curve is given by the equation $f_{b+1}^2 - 4 f_b f_{b+2}$. One can easily choose the polynomials $f_i$, $b \leq i \leq b+2$, so that the resulting plane curve is smooth, for example a Fermat type equation. We conclude that the genus of $D$ is at least 36 in the cases we are interested in. 
\end{enumerate}
 This completes the proof that $X$ is algebraically hyperbolic. 
	
\end{proof}

In summary, we see that $X$ is algebraically hyperbolic if $b = 0, a \geq 5$ or if $a \geq b+2 \geq 6$. It is not algebraically hyperbolic otherwise.

\subsection{$\PP(1,1,1,n)$}
Let $A = \PP(1,1,1,n)$, and let $X$ be a very general hypersurface in $A$ of class $\OO(mn)$. Then $A$ is the cone over the degree $n$ Veronese surface. The blowup $\tilde{A}$ of $A$ at the unique singular point is smooth, with Picard group generated by $H$ (the pullback of $\OO(n)$) and $F$ which is a divisor satisfying $nF \cong H-E$. The surface $X$ is smooth, with Picard group generated by $F$. We have the equations $K_{\tilde{A}} = -2H - (3-n)F$ and $K_X = (n(m-1)-3)F$.

\begin{lemma}
We have that $X$ will not be algebraically hyperbolic if $n = 1$, $m \leq 4$ or $m = 2$, $n \leq 4$.
\end{lemma}
\begin{proof}
Observe that if $n=1$, we are in the case of projective space. Hypersurfaces in projective space are algebraically hyperbolic when $m \geq 5$ and not if $m \leq 4$. If $m = 2$, then we are considering double covers of $\PP^2$ branched along a curve of degree $2n$. If $n =3$, then the bitangent lines to the branch locus have preimage consisting of a rational curve, and if $n=4$, the bitangent lines have preimage consisting of an elliptic curve, so neither is algebraically hyperbolic. 
\end{proof}

\begin{proposition}
$X$ is algebraically hyperbolic if $m \geq 4, n \geq 2$, $m= 3, n \geq 4$ or $m = 2, n \geq 5$.
\end{proposition}
\begin{proof}
First, we address the cases where $m \geq 3$. For these cases, observe that $H$ is a section-dominating line bundle for $\OO(mH)$. Thus, we have a map $M_H \to N_{f/X}$ whose image is not torsion. This implies that
\[ \deg N_{f/X} \geq -H \cdot C = -nF \cdot C. \]
Thus,
\[ 2g-2 \geq (K_X - nF) \cdot C = (n(m-2) - 3)F \cdot C . \]
If $m \geq 4, n \geq 2$ or $m = 3, n \geq 4$, we see that $n(m-2)-3$ is at least $1$, so algebraic hyperbolicity follows.

Next, consider the case $m=2$. Consider the bundle $M_F$. This bundle is not quite section dominating for $\OO(mH)$. If a section of $F$ vanishes at a point of $\tilde{A}$, it vanishes on the entire line through that point and the cone point (here we picture $A$ as the cone over a Veronese surface). There are two cases. First, suppose that some copy of $M_F$ maps nontrivially to the free part of $M_{L}$. Then we have
\[ \deg N_{f/X} \geq -F \cdot C  \]
so that 
\[ 2g-2 \geq (n(m-1)-4)F \cdot C . \]
If $m=2$, we see that $X$ is algebraically hyperbolic provided that $n \geq 5$.

Now suppose that none of the copies of $M_F$ map nontrivially to the free part of $N_{f/X}$. This means that at a general point $p$ of $\yY$, we have that $(T_{\yY}^{vert})_p$ contains all surfaces in $A$ with fixed intersection with a line passing through the exceptional divisor. If $m=2$, then this shows that the curve must be the ramification divisor of the double cover, which does not have genus $0$ or $1$ provided that $n$ is at least $5$, as required.
\end{proof}

The only remaining cases are $m=3$, $n=2,3$.

\bibliographystyle{plain}

\end{document}